\documentclass[10pt]{amsart}

\usepackage{amsfonts,amssymb}

\newtheorem{thm}{Theorem}[section]
\newtheorem{prop}[thm]{Proposition}

\newtheorem{lem}[thm]{Lemma}
\newtheorem*{thmA}{Theorem A}
\newtheorem*{thmB}{Theorem B}

\theoremstyle{remark}

\theoremstyle{definition}

\newcommand*\isom{%
  \xrightarrow{\sim}%
}

\def\pp{\mathbb{P}}

\def\zz{\mathbb{Z}}
\def\cc{\mathbb{C}}

\def\mm{\mathcal{M}}
\def\aa{\mathcal{A}}

\def\CC{\mathcal{C}}
\def\jj{\mathcal{J}}

\def\uu{\mathcal{U}}
\def\hh{\mathcal{H}}

\def\deldelbar{\partial \overline{\partial}}

\numberwithin{equation}{section}

\begin{document}

\title[Special values of canonical Green's functions]{Special values of canonical Green's functions}

\author{Robin de Jong}



\begin{abstract} We give a precise formula for the value of the canonical Green's function at a pair of Weierstrass points on a hyperelliptic Riemann surface. Further we express the `energy' of the Weierstrass points in terms of a spectral invariant recently introduced by N. Kawazumi and S. Zhang. It follows that the energy is strictly larger than $\log 2$. Our results generalize known formulas for elliptic curves.
\end{abstract}

\maketitle

\thispagestyle{empty}

\section{Introduction}

Let $M$ be a compact and connected Riemann surface of genus $h \geq 1$. On $M$ one has a canonical K\"ahler form $\mu$ given as follows. The space $H$ of holomorphic differentials on $M$ is a complex vector space of dimension $h$. It carries a natural hermitian inner product given by $\langle \alpha, \beta \rangle = \frac{\sqrt{-1}}{2} \int_M \alpha \wedge \overline{\beta}$ for elements $\alpha, \beta$ in $H$. Let $(\omega_1,\ldots,\omega_h)$ be an orthonormal basis of $H$. By the Riemann-Roch theorem, the $\omega_i$ do not simultaneously vanish at any point of $M$, and we put
\[ \mu = \frac{\sqrt{-1}}{2h} \sum_{i=1}^h \omega_i \wedge \overline{\omega}_i \, . \]
The associated metric has been studied in many contexts, ranging from arithmetic geometry \cite{ar} \cite{ce} \cite{fa} \cite{la} \cite{so} to perturbative string theory \cite{abmnv} \cite{amv} \cite{bm} \cite{bk} \cite{dp}. The metric is known to have everywhere non-positive curvature, and the curvature vanishes at $x \in M$ if and only if $x$ is a Weierstrass point on a hyperelliptic Riemann surface \cite{lew}, Main Theorem.

In this paper we study the canonical Green's function $g_\mu$ associated to $\mu$ (see Section \ref{green} for definitions), and in particular its special values $g_\mu(w_i,w_j)$ at pairs $(w_i,w_j)$ of distinct Weierstrass points on a hyperelliptic Riemann surface. The starting point of our discussion is the case $h=1$, with $M$ given as a $2$-sheeted cover of the Riemann sphere $\cc \pp^1$ branched at four points, say $\alpha_1,\alpha_2,\alpha_3$ and $\infty$. Let $w_1,w_2,w_3,o$ be the four critical points of $M$ lying over $\alpha_1,\alpha_2,\alpha_3,\infty$. It is then known \cite{br} \cite{djell} \cite{sz} that the formula
\begin{equation} \label{ellipticspecial}
g_\mu(w_i,w_j) = \frac{1}{3} \log 2 + \frac{1}{12} \log \left( \frac{ |\alpha_i-\alpha_j|^2}{|\alpha_i-\alpha_k | \cdot |\alpha_j - \alpha_k | } \right)
\end{equation}
holds. By the translation invariance of $g_\mu$ we obtain the formula
\begin{equation} \label{ellipticsum}
\sum_{i=1}^3 g_\mu(w_i,o) = \log 2 \, .
\end{equation}
We will return to both these formulas in the text below.

Our purpose is to generalize formulas (\ref{ellipticspecial}) and (\ref{ellipticsum}) to the case of a hyperelliptic Riemann surface of genus $h \geq 2$. Let $M$ be such a surface, given as the $2$-sheeted cover of $\cc \pp^1$ branched at the $2h+2$ points $\alpha_1,\ldots,\alpha_{2h+2}$ (these may or may not include $\infty$). We put, for each pair $(i,j)$ of distinct indices,
\begin{equation} \label{delta} \delta_{ij} = \frac{ (\alpha_i - \alpha_j )^{2h(2h+1)} \prod_{r \neq s} (\alpha_r - \alpha_s) }{ \prod_{r \neq i} (\alpha_i - \alpha_r)^{2h+1} \prod_{r \neq j} (\alpha_j - \alpha_r)^{2h+1} } \, .
\end{equation}
In this formula, we disregard any difference of two roots when one of the roots is $\infty$. It is readily verified that $\delta_{ij}$ is invariant under the action of $\mathrm{Aut}(\cc \pp^1)=\mathrm{PGL}_2(\cc)$ on the $\alpha_i$. In particular $\delta_{ij}$ defines an analytic modular invariant of the pair $(w_i,w_j)$, where $w_i,w_j$ are the critical points of $M$ lying over $\alpha_i,\alpha_j$. More intrinsically $\delta_{ij}$ is the discriminant of the branch set that one obtains by sending $\alpha_i,\alpha_j$ to $0,\infty$ using an element of $\mathrm{Aut}(\cc \pp^1)$ and by normalizing the other $2h$ roots such that their product is unity \cite{gu}.

We note that there is a tight connection with the more familiar cross ratios on the $\alpha_i$. Put $\eta_{ijk} = \delta_{ik} \delta_{jk}^{-1}$, and let
\[ \mu_{ijkr} = (\alpha_i - \alpha_k)(\alpha_j-\alpha_r)(\alpha_j-\alpha_k)^{-1}(\alpha_i-\alpha_r)^{-1} \]
be the cross ratio of the $4$-tuple $(\alpha_i,\alpha_j,\alpha_k,\alpha_r)$. Then the relation
\[ \mu_{ijkr}^{2h(2h+1)} = \eta_{ijk} \eta_{ijr}^{-1} = \delta_{ik} \delta_{jk}^{-1} \delta_{jr} \delta_{ir}^{-1} \]
holds. It is straightforward to verify that for each index $i$, the product $\prod_{j \neq i} \delta_{ij}$ is equal to unity.

The \emph{energy} of the Weierstrass points on $M$ is defined to be the following invariant of $M$:
\[ \psi(M) = \frac{1}{2h+2} \sum_{i \neq j} g_\mu(w_i,w_j) \, , \]
where the summation runs over all pairs $i,j$ of distinct indices.
Our first result is the following.
\begin{thmA} Let $w_i,w_j$ be two distinct Weierstrass points of $M$. Then the formula
\[ g_\mu(w_i,w_j) =  \frac{1}{4h(2h+1)} \log |\delta_{ij}| + \frac{1}{2h+1} \psi(M) \]
holds.
\end{thmA}
In our second result we give an expression for $\psi(M)$. Consider for the moment an arbitrary compact and connected Riemann surface $M$ of genus $h \geq 1$. Let $\Delta_\mu$ be the Laplacian on $L^2(M,\mu)$ given by putting $\partial \overline{\partial} \, f = \pi \sqrt{-1} \, \Delta_\mu(f) \, \mu$ for $C^\infty$ functions in $L^2(M,\mu)$. Let $(\phi_\ell)_{\ell=0}^\infty$ be an orthonormal basis of real eigenfunctions of $\Delta_\mu$, with corresponding eigenvalues $0=\lambda_0<\lambda_1\leq \lambda_2 \leq \ldots$. We then let $\varphi(M)$ denote the invariant
\[ \varphi(M) = \sum_{\ell>0} \frac{2}{\lambda_\ell} \sum_{m,n=1}^h \left| \int_M \phi_\ell \, \omega_m \wedge \bar{\omega}_n \right|^2  \]
of $(M,\mu)$. This fundamental invariant was introduced and studied recently by N.~Kawazumi \cite{kaw} and S. Zhang \cite{zh}. Note that $\varphi(M) \geq 0$, and that $\varphi(M)=0$ if and only if $M$ has genus one. In \cite{djasympt} the asymptotic behavior of the $\varphi$-invariant is determined for surfaces degenerating into a stable curve with a single node. The result shows in particular that $\varphi$ can be viewed as a Weil function, with respect to the boundary, on the stable (Deligne-Mumford) compactification of the moduli space of surfaces of a fixed genus $h \geq 2$.

\begin{thmB} Let $M$ be a hyperelliptic Riemann surface of genus $h \geq 2$, and let $\varphi(M)$ be the Kawazumi-Zhang invariant of $M$. Then for the energy $\psi(M)$ of the set of Weierstrass points of $M$, the formula
\[ \psi(M) = \frac{1}{2h} \, \varphi(M) + \log 2 \]
holds.
\end{thmB}
In particular we find that $\psi(M) > \log 2$. With $h=1$, the formulas in Theorems A and B specialize to equations (\ref{ellipticspecial}) and (\ref{ellipticsum}), respectively.

In order to prove Theorem B we proceed as follows. First we prove that, when viewed as functions on the moduli space of hyperelliptic Riemann surfaces, both $\psi$ and $\frac{1}{2h}\varphi$ have the same image under $\deldelbar$. Then we study the asymptotic behavior of the difference of $\psi$ and $\frac{1}{2h}\varphi$ on a family of hyperelliptic Riemann surfaces degenerating into a stable curve of compact type. By using induction on $h$ we will then deduce that the difference is a constant on the moduli space, equal to $\log 2$. For the degeneration argument we will rely heavily on results obtained by R. Wentworth \cite{we}.

\section{Canonical Green's functions} \label{green}

In this section we collect some basic results on canonical Green's functions. The main sources are \cite{ar} \cite{fa}. Proposition \ref{flat}, about a `parallel' property of certain $(1,1)$-forms related to the universal Riemann surface over the moduli space of Riemann surfaces, seems not to be well known, and is possibly of independent interest.

As before, let $M$ be a compact and connected Riemann surface of genus $h \geq 1$ and $\mu$ its canonical K\"ahler form. Let $\Delta$ be the diagonal on the complex manifold $M \times M$. The canonical Green's function $g_\mu=g_\mu(x,y)$ is the $C^\infty$-function on $M \times M \setminus \Delta$ uniquely characterized by the following conditions:
\begin{enumerate}
\item[(1)] $\partial \overline{\partial} g_\mu(x,y) =\pi \sqrt{-1}(\mu(y) - \delta_x)$ for all $x \in M$;
\item[(2)] $\int_M g_\mu(x,y) \mu(y) = 0$ for all $x \in M$;
\item[(3)] $g_\mu(x,y)=g_\mu(y,x)$ for all $x \neq y \in M$;
\item[(4)] $g_\mu(x,y) - \log|z(x)-z(y)|$ is bounded for all $x \neq y$ in a coordinate chart $(U,z \colon U \isom D)$ of $M$, where $D$ is the open unit disk.
\end{enumerate}
The canonical Green's function inverts the Laplacian $\Delta_\mu$ in the sense that
\[ f(x) = -\int_M g_\mu(x,y) \, \Delta_\mu(f)(y) \, \mu(y) + \int_M f \, \mu \]
holds for all $x \in M$ and all $f$ in $C^\infty(M)$. As a consequence we have a formal development
\[ g_\mu(x,y) = \sum_{\ell>0} \frac{ \phi_\ell(x) \phi_\ell(y)}{\lambda_\ell} \]
for $g_\mu$, where $\phi_\ell$ and $\lambda_\ell$ are the eigenfunctions and eigenvalues of $\Delta_\mu$ that we have introduced above.

In \cite{fa} Section 7 using standard harmonic analysis explicit formulas are derived for the $\phi_\ell$ and $\lambda_\ell$ in the case $h=1$. As a corollary of these formulas \cite{fa} Section 7 contains a derivation of equation (\ref{ellipticsum}) from the resulting formal expression for $g_\mu$. We would like to present here an alternative argument.
\begin{prop} \label{elliptic}
Assume $(M,o)$ is a complex torus, and let $N$ be a positive integer. Let $x_1,\ldots,x_{N^2-1}$ be the non-trivial $N$-torsion points of $M$. Then the formula
\[ \sum_{i=1}^{N^2-1} g_\mu(x_i,o) = \log N \]
holds.
\end{prop}
\begin{proof} From properties (1) and (2) it follows that for every $x,y$ on $M$ with $Nx \neq y$ the equality
\[ g_\mu(Nx,y) = \sum_{w \colon Nw =y} g_\mu(x,w) \]
holds. Now choose $x$ in a standard euclidean coordinate chart $z \colon U \isom D$ around ~$o$, where $D$ is the open unit disk. Then we have
\begin{align*}
g_\mu(Nx,o)  & =  \sum_{w \colon Nw=o} g_\mu(x,w) \\
             & =  \sum_{i=1}^{N^2-1} g_\mu(x_i,x) + g_\mu(x,o) \\
             & =  \sum_{i=1}^{N^2-1} g_\mu(x_i,x) + \log |z(x)| + a + o(1)
\end{align*}
as $x \to o$, where $a$ is some constant, by properties (3) and (4). On the other hand we have
\begin{align*}
g_\mu(Nx,o)  & =  \log |N z(x)| + a + o(1) \\
             & =  \log N + \log |z(x)| + a + o(1)
\end{align*}
as $x \to o$, by property (4). The desired equality follows.
\end{proof}
One obtains (\ref{ellipticsum}) by taking $N=2$ in the above proposition.

Let $\mm$ denote the moduli space of compact and connected Riemann surfaces of genus $h \geq2$, and denote by $\pi \colon \CC \to \mm$ the universal surface over $\mm$. Both are viewed as orbifolds. The canonical Green's functions on the fibers of $\pi$ determine a generalized function $g$ on the fiber product $\CC \times_\mm \CC$ with logarithmic singularities along the diagonal $\Delta$. Let $\nu$ be the $(1,1)$-form on $\CC \times_\mm \CC$ determined by the condition $\partial \overline{\partial} g = \pi \sqrt{-1}(\nu - \delta_\Delta)$. Let $e^A$ denote the restriction of $\nu$ to $\Delta$, which we view as another copy of $\CC$. By \cite{ar} the form $e^A$ restricts to $2-2h$ times the canonical $(1,1)$-form $\mu$ in each fiber of $\pi \colon \CC \to \mm$.

The $(1,1)$-forms $e^A$ and $\nu$ are parallel over $\mm$ in the following sense.
\begin{prop} \label{flat} Let $s,t$ be arbitrary holomorphic sections of $\pi$ over an open subset $U$ of $\mm$. Then the equalities $s^*e^A =t^*e^A=(s,t)^*\nu$ hold on $U$.
\end{prop}
\begin{proof} Without loss of generality we may replace $\mm$ by a finite cover so that we may assume that $\mm$ is equipped with a universal theta characteristic $\alpha$, i.e. a consistent choice of a divisor class $\alpha$ of degree $h-1$ on each surface $M$ such that $2 \alpha$ is the canonical class. Let $\jj \to \mm$ be the universal jacobian over $\mm$ and let $\Phi \colon \CC \times_\mm \CC \to \jj$ be the map over $\mm$ given by sending a triple $(M,x,y)$ to the pair consisting of the jacobian $J(M)$ of $M$ and the class of the divisor $h \cdot x -y - \alpha$ in $J(M)$. Let $\bar{\pi} \colon \CC \times_\mm \CC \to \CC$ be the projection on the first factor and let $\bar{s},\bar{t} \colon \pi^{-1}U \to \CC \times_\mm \CC$ be the local sections of $\bar{\pi}$ obtained by pulling back the given local sections $s,t$ of $\pi$. It follows from \cite{dj} Lemma~3.2 that there exists a $(1,1)$-form $\kappa$ on $\CC$ and a $(1,1)$-form $w$ on $\jj$ which is translation-invariant in the fibers of $\jj \to \mm$, such that $\Phi^*w = h \, \nu + \bar{\pi}^* \kappa$ holds. We deduce from this that
\[ h \, \bar{s}^* \nu = (\Phi \bar{s})^*w - \kappa \quad \textrm{and} \quad
h \, \bar{t}^* \nu = (\Phi \bar{t})^*w - \kappa \, . \]
Let $T_{x,y}$ denote translation by $[x-y]$ over $\jj|_U$. Then $\Phi\bar{t}=T_{s,t}\Phi \bar{s}$ and since $w$ is fiberwise translation-invariant we obtain $\bar{s}^* \nu = \bar{t}^* \nu$ from the above equalities. We derive $(s,t)^*\nu = s^* \bar{t}^* \nu = s^* \bar{s}^* \nu = (s,s)^*\nu = s^*e^A$. By symmetry property (3) we have $(s,t)^*\nu = (t,s)^*\nu = t^* e^A$ as well.
\end{proof}

\section{Proof of Theorem A}

In this section we give a proof of Theorem A. Let $x \colon M \to \cc \pp^1$ be the $2$-sheeted cover with branch points $\alpha_1,\ldots,\alpha_{2h+2}$. Let $w_1,\ldots,w_{2h+2}$ be the corresponding Weierstrass points on $M$. For the moment we fix two distinct indices $i,j$. Consider then the meromorphic function $f_{ij}=(x-\alpha_i)(x-\alpha_j)^{-1}$ on $M$. We note that $\mathrm{div}(f_{ij}) = 2(w_i-w_j)$. From properties (1) and (2) we therefore obtain an equality
\[ g_\mu(w_i,z)-g_\mu(w_j,z)= \frac{1}{2} \log |f_{ij}(z)| - \frac{1}{2} \int_M \log |f_{ij}| \cdot \mu \]
of generalized functions on $M$. In particular we find for any pair of indices $k,r$
\begin{align*}
g_\mu(w_i,w_k)-g_\mu(w_i,w_r)-&g_\mu(w_j,w_k)+g_\mu(w_j,w_r) \\
= & \frac{1}{2} \log | f_{ij}(w_k)f_{ij}(w_r)^{-1}| \\
= & \frac{1}{2} \log |\mu_{ijkr}| \\
= & \frac{1}{4h(2h+1)} \log |\eta_{ijk} \eta_{ijr}^{-1} |
\end{align*}
where we recall that $\eta_{ijk}=\delta_{ik}\delta_{jk}^{-1}$ with $\delta_{ij}$ the expression from (\ref{delta}). This equality implies that there exists a constant $c_{ij}$ such that
\[ g_\mu(w_i,w_k) - g_\mu(w_j,w_k) = \frac{1}{4h(2h+1)} \log |\eta_{ijk}| + c_{ij}  \]
for each index $k$. It follows from \cite{djhyp}, Theorem~1.4 that
\[ \sum_{k \neq i} g_\mu(w_i,w_k) = \sum_{k \neq j} g_\mu(w_j,w_k) \, . \]
From the fact that $\prod_{k \neq i,j} \delta_{ij}$ equals unity we obtain that $\sum_{k \neq i,j} \log |\eta_{ijk}| = 0$. It follows that
\[ 2h \, c_{ij} =\sum_{k \neq i,j} (g_\mu(w_i,w_k)-g_\mu(w_j,w_k)) = 0 \, , \]
hence $c_{ij}=0$ and
\[ g_\mu(w_i,w_k)-g_\mu(w_j,w_k) = \frac{1}{4h(2h+1)} \log|\eta_{ijk}| = \frac{1}{4h(2h+1)} \log |\delta_{ik} \delta_{jk}^{-1}| \, . \]
Now fix the index $k$ and vary the indices $i,j$. We find that there exists a constant $c_k$ such that
\[ g_\mu(w_i,w_k) = \frac{1}{4h(2h+1)} \log |\delta_{ik}| + c_k  \]
for all indices $i \neq k$. We obtain $c_k = \frac{1}{2h+1}\psi(M)$ by summing over $i \neq k$. The formula in Theorem A follows.

Note that a similar argument in the case $h=1$ allows one to deduce (\ref{ellipticspecial}) from (\ref{ellipticsum}).

\section{Proof of Theorem B}

As was announced in the Introduction, we proceed in several steps. Let $\hh$ be the moduli space of hyperelliptic Riemann surfaces of a fixed genus $h \geq 2$. Let $\pi \colon \CC \to \hh$ be the universal surface over $\hh$. Both are viewed as orbifolds. A first result is the following.
\begin{prop} \label{pluriharmonic} The equality of $(1,1)$-forms
\[ \partial \overline{\partial} \, \psi = \frac{1}{2h} \partial \overline{\partial} \, \varphi \]
holds on $\hh$.
\end{prop}
\begin{proof} Let $\uu$ be an open cover of $\hh$ such that for each $U$ in $\uu$ there exist $2h+2$ holomorphic sections $w_1,\ldots,w_{2h+2} \colon U \to \pi^{-1}U$ of $\pi^{-1}U \to U$ such that each $w_i$ is a Weierstrass point in each fiber of $\pi^{-1}U \to U$. Let $U$ be an element of $\uu$. It suffices to prove the required equality on $U$. Theorem 3.1 of \cite{kaw} gives an expression for $\partial \overline{\partial} \, \varphi$ over the moduli space of Riemann surfaces of genus $h$ (note that \cite{kaw} considers the invariant $a$ with $a=\frac{1}{2\pi}\varphi$). Since the `harmonic volume' of a Weierstrass point on a hyperelliptic Riemann surface is trivial, this expression immediately implies that on $U$ one has the equality
\[ \partial \overline{\partial} \, \varphi = 2h(2h+1)\pi \sqrt{-1} \, w_i^* e^A \, , \]
for any choice of the index $i$. On the other hand by Proposition \ref{flat} we have
\[ \partial \overline{\partial} \, \psi = \pi \sqrt{-1} \sum_{j \neq i} (w_i,w_j)^* \nu = (2h+1)\pi \sqrt{-1} \, w_i^* e^A  \]
for each index $i$. The proposition follows.
\end{proof}
We next determine the limit behavior of the difference $\psi - \frac{1}{2h}\varphi$ in a holomorphic family $M_t$ of hyperelliptic Riemann surfaces of genus $h$ degenerating into a stable curve $M_0$ which is the union of two hyperelliptic Riemann surfaces $M_1,M_2$ of genera $h_1,h_2 \geq 1$, respectively, with two points $p_1 \in M_1$ and $p_2 \in M_2$ identified, as in Chapter III of the book `Theta functions on Riemann surfaces' by J. Fay \cite{fay}. Here $t$ runs through a small punctured open disk around $0$ in the complex plane. Note that $h=h_1+h_2$. Let $g_1,g_2$ be the canonical Green's functions on $M_1,M_2$, respectively, and let $z_1,z_2$ be the local coordinates around $p_1,p_2$ on $M_1,M_2$ that come with the degeneration model. We put, following \cite{we}, Section~6:
\[ \log k_1 = \lim_{x \to p_1} \left[ g_1(x,p_1) - \log|z_1(x)| \right] \, , \quad
\log k_2 = \lim_{x \to p_2} \left[ g_2(x,p_2) - \log|z_2(x)| \right] \, , \]
and then make the reparametrization $\tau=k_1k_2t$.
\begin{thm} \label{finalphi} For the surfaces $M_t$ in Fay's degeneration model, the limit formula
\[ \lim_{t\to 0} \left[ \frac{1}{2h} \, \varphi(M_t) + \frac{h_1h_2}{h^2} \log |\tau| \right] = \frac{1}{2h} \, \varphi(M_1) + \frac{1}{2h} \, \varphi(M_2) \]
holds.
\end{thm}
\begin{thm} \label{finalpsi}
For the energy $\psi(M_t)$ of the Weierstrass points of $M_t$, the limit formula
\[ \lim_{t\to 0} \left[ \psi(M_t) + \frac{h_1h_2}{h^2} \log |\tau| \right] = \frac{h_1}{h} \, \psi(M_1) + \frac{h_2}{h} \, \psi(M_2) \]
holds.
\end{thm}
We deduce Theorem B from Theorems \ref{finalphi} and \ref{finalpsi} by induction on $h$. Note that the assertion in Theorem B is true for $h=1$ by (\ref{ellipticsum}) and the fact that $\varphi$ vanishes for $h=1$. Now assume that the assertion in Theorem B is true for all genera smaller than $h$. Consider a holomorphic family $M_t$ of hyperelliptic Riemann surfaces of genus $h$ as above. We find
\begin{align*}
\lim_{t\to 0} \left[ \psi(M_t) - \frac{1}{2h}\varphi(M_t) \right] = & \frac{h_1}{h}\psi(M_1) - \frac{1}{2h} \varphi(M_1) + \frac{h_2}{h} \psi(M_2) - \frac{1}{2h} \varphi(M_2) \\
= & \frac{h_1}{h}\left( \psi(M_1)-\frac{1}{2h_1}\varphi(M_1)\right) +
\frac{h_2}{h}\left( \psi(M_2)-\frac{1}{2h_2}\varphi(M_2)\right) \\
= & \frac{h_1}{h} \log 2 + \frac{h_2}{h}\log 2 = \log 2 \, .
\end{align*}
Let $\overline{\hh}$ be the closure of $\hh$ in the Deligne-Mumford compactification $\overline{\mm}$ of the moduli space $\mm$ of compact Riemann surfaces of genus $h$.
Let $\widetilde{\hh} \subset \overline{\hh}$ be the partial compactification of $\hh$ given by stable curves of compact type (i.e. whose jacobians are principally polarized abelian varieties). Write $\chi = \psi-\frac{1}{2h}\varphi$. The above calculation implies that for hyperelliptic Riemann surfaces in $\hh$ tending to a point in the boundary of $\hh$ in $\widetilde{\hh}$, the function $\chi$ tends to $\log 2$. In particular the function $\chi$ extends as a continuous function over $\widetilde{\hh}$. By Proposition \ref{pluriharmonic} we know that $\chi$ is a pluriharmonic function on $\hh$. By the Riemann extension theorem we find that $\chi$ extends as a pluriharmonic function on $\widetilde{\hh}$, which is a constant function with value $\log 2$ on $\widetilde{\hh} \setminus \hh$.

Let $\aa^*$ be the Satake compactification of the moduli space $\aa$ of principally polarized abelian varieties of genus $h$. This is a normal projective variety \cite{bb}.  There exists a holomorphic map $t \colon \overline{\hh} \to \aa^*$ given by assigning to a stable curve $M$ the jacobian of its normalization. The map $t$ is injective on $\hh$, and the function $\chi$ descends as a pluriharmonic function on $t(\widetilde{\hh})$, which is a constant function with value $\log 2$ on $t(\widetilde{\hh} \setminus \hh)$.

The boundary of $\widetilde{\hh}$ in $\overline{\hh}$ is a finite union of divisors $\Xi_i$, see \cite{ch} for example. From the moduli theoretic description of the points in the divisors $\Xi_i$ it follows that the map $t$ restricted to the boundary has positive dimensional fibers. Hence the boundary of $t(\widetilde{\hh})$ in the closure of $t(\widetilde{\hh})$ in $\aa^*$ has codimension at least two. Let $M$ be a point in $t(\hh)$. Using that $\aa^*$ is projective one sees from intersecting with suitable hyperplanes in a projective embedding that the variety $t(\widetilde{\hh})$ contains a complete curve $X$ passing through $M$. As $t(\hh)$ is affine, the curve $X$ necessarily contains a point of $t(\widetilde{\hh} \setminus \hh)$. Let $Y \to X$ be the normalization of $X$. The function $\chi$ pulls back to a pluriharmonic function on $Y$, which is then necessarily constant. It follows that the value of $\chi$ at $M$ is equal to $\log 2$, and Theorem B is proven.

For the proof of Theorem \ref{finalphi} we refer to \cite{djasympt}. It remains to prove Theorem \ref{finalpsi}. We need the next result, which follows from \cite{we}, Theorem 6.10.
\begin{thm} \label{wentworth} Let $g_t$ be the canonical Green's function on $M_t$. Denote by $p$ both the point $p_1$ on $M_1$ and the point $p_2$ on $M_2$. Then for local distinct holomorphic sections $x,y$ of the family $M_t$ specializing onto $M_1$ we have
\[ \lim_{t\to 0} \left[ g_t(x,y) - \left( \frac{h_2}{h} \right)^2 \log |\tau| \right] = g_1(x,y) - \frac{h_2}{h}\left(g_1(x,p) + g_1(y,p) \right) \, . \]
For local distinct holomorphic sections $x,y$ with $x$ specializing onto $M_1$ and $y$ specializing onto $M_2$ we have
\[ \lim_{t \to 0} \left[ g_t(x,y) + \frac{h_1h_2}{h^2} \log |\tau| \right] = \frac{h_1}{h}g_1(x,p) + \frac{h_2}{h} g_2(y,p) \, . \]
\end{thm}
Now let $w_1,\ldots,w_{2h+2}$ be the Weierstrass points on $M_t$. As $t \to 0$ a portion of $2h_1+1$ of them degenerate onto $M_1$, and $2h_2+1$ of them degenerate onto $M_2$. The point $p$ is a Weierstrass point of both $M_1,M_2$. Let $I$ be the set of indices such that $i \in I$ if and only if $w_i$ degenerates onto $M_1$.
\begin{lem} \label{small} The equality
\[ \sum_{i,j \in I \atop i \neq j} g_1(w_i,w_j) = 2h_1 \, \psi(M_1) \]
holds.
\end{lem}
\begin{proof} We calculate
\[ (2h_1+2) \, \psi(M_1) = \sum_{i,j \in I \cup \{p\} \atop i \neq j} g_1(w_i,w_j) =
\sum_{i,j \in I \atop i \neq j} g_1(w_i,w_j) + 2 \, \psi(M_1) \, . \]
The required equality follows.
\end{proof}
Using Theorem \ref{wentworth} we find
\begin{align*}
  (2h+2) \, \psi(M_t)   =  & \sum_{i \neq j} g_t(w_i,w_j)  \\
   =  &  \sum_{i,j \in I \atop i \neq j} g_t(w_i,w_j) + \sum_{i,j \in I^c \atop i \neq j} g_t(w_i,w_j) + 2 \sum_{i \in I \atop j \in I^c} g_t(w_i,w_j)  \\
  =& \sum_{i, j \in I \atop i \neq j} \left( \left( \frac{h_2}{h} \right)^2 \log|\tau| + g_1(w_i,w_j) - \frac{h_2}{h} \left( g_1(w_i,p) + g_1 (w_j,p) \right) \right) \\
  & + \sum_{i,j \in I^c \atop i \neq j} \left( \left( \frac{h_1}{h} \right)^2 \log |\tau| + g_2(w_i,w_j) - \frac{h_1}{h} \left(g_2(w_i,p) + g_2(w_j,p) \right) \right) \\
  & + 2 \sum_{i \in I \atop j \in I^c} \left( -\frac{h_1 h_2}{h^2} \log |\tau| + \frac{h_1}{h} g_1(w_i,p) + \frac{h_2}{h} g_2(w_j,p) \right) + o(1)
\end{align*}
as $t \to 0$. With the help of Lemma \ref{small} this can be rewritten as
\begin{align*}
 (2h+2) \, \psi(M_t) =&
    \left( (2h_1+1)\, 2h_1 \left( \frac{h_2}{h} \right)^2 + (2h_2+1) \, 2h_2 \left( \frac{h_1}{h} \right)^2 \right)\log |\tau| \\
 & - 2(2h_1+1)(2h_2+1) \frac{h_1h_2}{h^2} \, \log |\tau| \\
&  + \left( 2h_1 - \frac{4h_1h_2}{h} + \frac{2h_1(2h_2+1)}{h} \right) \psi(M_1) \\
&  + \left( 2h_2 - \frac{4h_1h_2}{h} + \frac{2h_2(2h_1+1)}{h} \right) \psi(M_2) + o(1) \\
= &    (2h+2)\left(-\frac{h_1h_2}{h^2} \log |\tau| + \frac{h_1}{h}\psi(M_1) + \frac{h_2}{h} \psi(M_2) \right) + o(1)
\end{align*}
as $t \to 0$. Theorem \ref{finalpsi} follows.

\vspace{0.5cm}

\noindent Address of the author:\\ \\
Mathematical Institute \\
University of Leiden \\
PO Box 9512 \\
2300 RA Leiden \\
The Netherlands \\
Email: \verb+rdejong@math.leidenuniv.nl+

\end{document}